\documentclass[11pt]{amsart}

\usepackage{geometry}
\usepackage{todonotes}

\usepackage{amsfonts, amssymb, amscd}
\numberwithin{equation}{section}

\usepackage[symbol]{footmisc}

\usepackage{bm}
\usepackage{verbatim}
\usepackage{mathrsfs}
\usepackage{graphicx}
\usepackage{tikz-cd}
\usepackage{subcaption}
\usepackage{listings}
\usepackage{subfiles}
\usepackage[toc,page]{appendix}
\usepackage{mathtools}
\usepackage{comment}
\usepackage{enumerate}
\usepackage{enumitem}
\usepackage[all]{xy}

\usepackage{graphicx}
\graphicspath{{images/}}

\usepackage{appendix}
\usepackage{hyperref}
\hypersetup{
	colorlinks=true,
	citecolor=red,
	linkcolor=blue,
	filecolor=magenta,      
	urlcolor=red,
}
\newcommand{\Cd}{\mathcal{D}}
\newcommand{\Ci}{\mathcal{I}}
\newcommand{\Co}{\mathcal{O}}

\newcommand{\Cs}{\mathcal{S}}
\newcommand{\Cx}{\mathcal{X}}
\newcommand{\Cy}{\mathcal{Y}}
\newcommand{\Cn}{\mathcal{N}}
\newcommand{\Cl}{\mathcal{L}}
\newcommand{\Cb}{\mathcal{B}}
\newcommand{\Ce}{\mathcal{E}}
\newcommand{\Cw}{\mathcal{W}}
\newcommand{\Cv}{\mathcal{V}}
\newcommand{\Ct}{\mathcal{T}}

\newcommand{\Supp}{\mathrm{Supp}}
\newcommand{\vol}{\mathrm{vol}}

\theoremstyle{plain} 
\newtheorem{thm}{Theorem}[section] 
\newtheorem{lemma}[thm]{Lemma}
\newtheorem{prop}[thm]{Proposition}
\newtheorem{cor}[thm]{Corollary}
\theoremstyle{definition} 
\newtheorem{defn}[thm]{Definition} 

\theoremstyle{remark} 
\newtheorem{rem}[thm]{Remark}

\begin{document}

	\title{Locally stable degenerations of log Calabi-Yau pairs}
	\author{Junpeng Jiao}

	\address{Yau Mathematical Sciences Center, Tsinghua University, Beijing, China}
	\email{jiao$\_$jp@tsinghua.edu.cn}
	
	\date{\today}

	\begin{abstract}
		We study the birational boundedness of special fibers of log Calabi-Yau fibrations and Fano fibrations. We show that for a locally stable family of Fano varieties or polarised log Calabi-Yau pairs over a curve, if the general fiber satisfies some natural boundedness conditions, then every irreducible component of the special fiber is birationally bounded. 
	\end{abstract}
	
	\maketitle
	Throughout this paper, we work over the complex number field $\mathbb{C}$.
	\section{Introduction}
	
	The notion of locally stable morphisms, firstly defined for surfaces in \cite{KSB88} and defined for all dimensions in \cite{Kol23}, is very important in the moduli of varieties. By a locally stable degeneration of a set of varieties $\mathscr{C}$, we mean the special fiber of a locally stable morphism whose general fiber is in $\mathscr{C}$. It is natural to ask if $\mathscr{C}$ is a bounded family, then are all of their locally stable degenerations in a bounded family?
	
	In this paper, we use singularities of log pairs and moduli of polarised log Calabi-Yau pairs to study the birational boundedness of special fibers of Fano fibrations and log Calabi-Yau fibrations. The boundedness of $\epsilon$-lc Fano varieties is studied in \cite{Bir19} and \cite{Bir21} and the boundedness of $(d,c,v)$-polarised Calabi-Yau pairs is studied in \cite{Bir23}. The first result shows locally stable degeneration of these two class of pairs is birationally bounded (see Definition \ref{birationally bounded}).
	
	\begin{thm}\label{Main theorem, locally stable morphism}
		Fix a natural number $d$ and positive rational numbers $c,v,\epsilon$. Let $X$ be a quasi-projective normal variety, $f:(X,\Delta)\rightarrow C$ be a locally stable morphism of relative dimension $d$ over a smooth curve $C$. Suppose either
		\begin{enumerate}
			\item $-(K_X+\Delta)$ is ample over $C$, the general fiber $(X_g,\Delta_g)$ is $\epsilon$-lc and $\mathrm{coeff}(\Delta)\subset c\mathbb{N}$, or
			\item $K_X+\Delta\sim_{\mathbb{Q},C}0$ and there is a divisor $N$ on $X$ such that the general fiber $(X_g,\Delta_g),N_g$ is a $(d,c,v)$-polarised Calabi-Yau pair (see Definition \ref{definition of log Calabi-Yau pairs}).
		\end{enumerate}
		
		Then every irreducible component of $X_0$ is birationally bounded for any closed point $0\in C$.
	\end{thm}

	If the morphism is not locally stable, we have a weaker result.
	
	\begin{thm}\label{Main theorem, flat morphism}
		Fix a natural number $d$ and positive rational numbers $c,v,\epsilon$. Then there exists a natural number $l$ and a bounded family of projective varieties $\mathcal{W}\rightarrow \Ct$, such that:
		
		Let $X$ be a normal quasi-projective variety, $(X,\Delta)$ is a log canonical pair. $f:X\rightarrow C$ be a fibration of relative dimension $d$ over a smooth curve $C$. Suppose either
		\begin{enumerate}
			\item $-(K_X+\Delta)$ is ample over $C$ and the general fiber $(X_g,\Delta_g)$ is $\epsilon$-lc, or
			\item $K_X+\Delta\sim_{\mathbb{Q},C}0$ and there is a divisor $N$ on $X$ such that the general fiber $(X_g,\Delta_g),N_g$ is a $(d,c,v)$-polarised Calabi-Yau pair.
		\end{enumerate}
		Then for any closed point $0\in C$, if $P$ a log canonical place of $(X,\Delta+\mathrm{lct}(X,\Delta;f^*0)f^*0)$, there is a closed point $t\in \Ct$ and a finite dominant rational map $\Cw_t\dashrightarrow P$ whose degree is a factor of $\mathrm{min}\{l,\mathrm{mult}_{P}f^*0!\}$. ($\mathrm{mult}_{P}f^*0!$ is the factorial of $\mathrm{mult}_{P}f^*0$.)
		
	\end{thm}
	
	\noindent\textbf{Acknowledgement}. The author would like to thank his advisor Christopher D. Hacon and his postdoctoral advisor Caucher Birkar for their encouragement and constant support. He would also like to thank Jingjun Han, Chen Jiang, Xiaowei Jiang, and Jihao Liu for their helpful comments. This work was supported by grants from Tsinghua University, Yau Mathematical Science Center. 	
	
	\section{Preliminary}
	\subsection{Notations and basic definition.}  We will use the same notation as in \cite{KM98} and \cite{Laz04}.
	
	A sub-pair $(X,\Delta)$ consists of a normal quasi-projective variety $X$ and a $\mathbb{Q}$-divisor $\Delta$ such that $K_X+\Delta$ is $\mathbb{Q}$-Cartier.
	Let $(Y,\Delta_Y),(X,\Delta)$ be two sub-pairs and $h:Y\rightarrow X$ a birational morphism, we say $(Y,\Delta_Y)\rightarrow (X,\Delta)$ is a crepant birational morphism if $K_Y+\Delta_Y\sim_{\mathbb{Q}}h^*(K_X+\Delta)$, two sub-pairs $(X_i,\Delta_i),i=1,2$ are crepant birationally equivalent if there is a sub-pair $(Y,\Delta_Y)$ and two crepant birational morphisms $(Y,\Delta_Y)\rightarrow (X_i,\Delta_i),i=1,2$.
	
	A contraction is a projective morphism $f:X\rightarrow Z$ of quasi-projective varieties with $f_*\mathcal{O}_X=\mathcal{O}_Z$. If $X$ is normal, then so is $Z$, and the fibers of $f$ are connected.
	A fibration is a contraction $f:X\rightarrow Z$ of normal quasi-projective varieties with $\mathrm{dim}X>\mathrm{dim}Z$.
	
	In this paper we only consider locally stable morphisms over smooth varieties, the following definition comes from \cite[Corollary 4.55]{Kol23}.
	\begin{defn}
		
		Let $S$ be a smooth variety, $(X,\Delta)$ be a log canonical pair, and $f:X\rightarrow S$ be a morphism. We say that $f:(X,\Delta)\rightarrow S$ is locally stable if 
		the pair $(X,\Delta+f^*D)$ is semi-log-canonical for every snc divisor $D\subset S$.
		
		Since $S$ is smooth, $D$ is a Cartier divisor, this implies that $(X,\Delta)$ is semi-log-canonical.
		
	\end{defn}
	
	\begin{defn}\label{birationally bounded}
		We say that a set $\mathscr{X}$ of varieties is birationally bounded if there is a projective morphism $\Cw\rightarrow \Ct$, where $\Ct$ is of finite type, such that for every $X\in\mathscr{X}$, there is a closed point $t\in \Ct$ and a birational map $f:\Cw_t\dashrightarrow X$.
	\end{defn}

	\begin{thm}[{\cite[Theorem 1.9]{Bir23}}]\label{Birkar:log Calabi-Yau fibrations Theorem 1.10}
		Fix a natural number $d$ and a finite set of rational numbers $\Ci\subset [0,1]$. Then there exists a natural number $n$ depending only on $d,\Ci$ satisfying the following. Assume
		\begin{itemize}
			\item $(X,B)$ is a projective log canonical pair of dimension $d$,
			\item the coefficients of $B$ are in $\Ci$,
			\item $M$ is a semi-ample Cartier divisor on $X$ defining a contraction $f:X\rightarrow Z$,
			\item $X$ is of Fano type over $Z$,
			\item $M-(K_X+B)$ is nef and big, and
			\item $S$ is a non-klt centre of $(X,B)$ with $M|_{S}\equiv 0$.
		\end{itemize}
		Then there is a $\mathbb{Q}$-divisor $\Lambda \geq B$ such that
		\begin{itemize}
			\item $(X,\Lambda)$ is log canonical over a neighbourhood of $z:=f(S)$, and
			\item $n(K_X+\Lambda)\sim (n+2)M$.
		\end{itemize}
	\end{thm}

	\begin{lemma}\label{relative complement}
		Fix a natural number $d$, a positive rational number $c$, then there is a natural number $l$ depending only on $d,c$, such that if $f:X\rightarrow C$ is a fibration over a curve $C$, $(X,\Delta)$ is a log canonical pair such that
		\begin{itemize}
			\item $\mathrm{coeff} B\subset c\mathbb{N}$,
			\item $-(K_X+\Delta)$ is ample over $C$, and
			\item the general fiber $(X_g,\Delta_g)$ of $f$ is klt.
		\end{itemize}
		Then for any closed point $s\in C$ and any log canonical place $P$ of $(X,\Delta+\mathrm{lct}(X,\Delta;f^*(s))f^*(s))$, there exists a birational map $Y\dashrightarrow X$ and a $\mathbb{Q}$-divisor $\Lambda_Y$ satisfying the following properties:
		\begin{itemize}
			\item $l(K_Y+\Lambda_Y)\sim_C 0$.
			\item $(Y,\Lambda_Y)$ is log canonical near $s$.
			\item $P$ is a component of $\Lambda_Y^{=1}$.
			\item the general fiber $Y_g$ is isomorphic to $X_g$. 
		\end{itemize}
		\begin{proof}
			Let $g:X'\rightarrow X$ be a dlt modification of $(X,\Delta+\mathrm{lct}(X,\Delta;f^*(s))f^*(s))$ such that $P$ is a divisor on $X'$. Define $\Delta'$ and $D'$ by 
			$$K_{X'}+\Delta'\sim_{\mathbb{Q}}g^*(K_X+\Delta),\text{ and}$$
			$$K_{X'}+D'\sim_{\mathbb{Q}}g^*(K_X+\Delta+\mathrm{lct}(X,\Delta;f^*(s))f^*(s)).$$
			Because $-(K_X+\Delta)$ is ample over $C$, $-(K_{X'}+D')$ is semi-ample over $C$, we can choose a general member $B'\in |-(K_{X'}+D')/C|_{\mathbb{Q}}$   
			such that $(X',D'+B')$ is a dlt pair. 
			
			Since every $g$-exceptional divisor has log discrepancy 0 with respect to $(X',D')$, and every divisor on $X$ supported on $f^{-1}s$ has log discrepancy $<1$ with respect to $(X,\Delta+\mathrm{lct}(X,\Delta;f^*(s))f^*(s))$. Then $f'^{-1}s\subset \Supp D'$. Let $F':=\Supp D'^{<1} \cap f'^{-1}s$, where $f'$ denotes the natural morphism $f':X'\rightarrow C$. By \cite[Corollary 2.39]{KM98}, there exists a positive rational number $\delta\ll 1$ such that $(X',D'+\delta F'+B')$ is a dlt pair.
			
			Next, we run a $(K_{X'}+D'+\delta F'+B')$-MMP with scaling of an ample divisor. Because
			$$K_{X'}+D'+\delta F'+B'\sim_{\mathbb{Q},C}\delta F',$$
			by \cite[Theorem 1.8]{Bir12}, this MMP terminates with a model $f'':X''\rightarrow C$ such that $F''=0$ and $K_{X''}+D''+B''\sim_{\mathbb{Q},C}0$, where $D'',F''$ and $B''$ is the pushforward of $D', F'$ and $B'$. Because $\Supp F'=f'^{-1}s \cap \Supp D'^{<1}$ and $f'^{-1}s\subset \Supp D'$, we have $\mathrm{coeff}_QD''=1$ for every irreducible component $Q$ of $f''^{-1}s$. Notice that $\Supp P \not \subset \Supp F'$, $P$ is not contracted by $X'\dashrightarrow X''$.
			
			Notice that the general fiber of $(X,\Delta+B)$ is klt, for a general point $g\in C$, we have 
			$(X_g,\Delta_g)\cong (X'_g,D'_g)\cong (X''_g,D''_g)$, and $D''=\Delta''_h+\mathrm{red}(f''^{-1}s)$. 
			
			Because $(X'',D''+B'')$ is dlt and $\lfloor B''\rfloor=0$, by \cite[Corollary 2.39]{KM98}, we can choose $\delta'\ll 1$ such that $(X'',D''+(1+\delta'')B'')$ is a dlt pair, also since $\Supp D''\supset f''^{-1}s$, we may assume $(X'',D''-\delta' f''^*s+(1+\delta'')B'')$ is a klt pair. Because $B''$ is big over $C$, by the main theorem of \cite{BCHM10}, we can run a $(K_{X''}+D''-\delta' f''^*s+(1+\delta'')B'')$-MMP over $C$, which will terminates with a model $Y$ such that $K_Y+D_Y-\delta'f_Y^*s+(1+\delta')B_Y$ is nef and big over $C$, where $D_Y$ and $B_Y$ are the strict transform of $D''$ and $B''$, $f_Y$ is the natural morphism $Y\rightarrow C$.
			
			Since $K_Y+D_Y-\delta'f_Y^*s+(1+\delta')B_Y\sim_{\mathbb{Q},C}K_Y+D_Y+(1+\delta')B_Y\sim_{\mathbb{Q},0}\delta'B_Y\sim_{\mathbb{Q},C}-\delta'(K_Y+D_Y)$, $-(K_Y+D_Y)$ is nef and big over $C$. Also because $X''\dashrightarrow Y$ is a $(K_{X''}+D''+(1+\delta')B'')$-MMP, $(Y,D_Y+(1+\delta')B_Y)$ is dlt, hence $(Y,D_Y)$ is dlt. Notice that $X''\dashrightarrow Y$ is a $B''$-MMP and $\Supp P\not \subset \Supp B''$, $P$ is not contracted by $X''\dashrightarrow Y$.
			
			Because $D''=\Delta''_h+\mathrm{red}(f''^{-1}s)$, $\mathrm{coeff}D_Y\subset c\mathbb{N}$ and $f_Y^{-1}s\subset \Supp D_Y^{=1}$, then by Theorem \ref{Birkar:log Calabi-Yau fibrations Theorem 1.10}, there exists a $\mathbb{Q}$-divisor $\Lambda_Y\geq D_Y$ such that $l(K_Y+\Lambda_Y)\sim 0$ and $(Y,\Lambda_Y)$ is log canonical over a neighborhood of $s$. Because $P$ is a component of $D'^{=1}$ and not contracted by $X'\dashrightarrow X''\dashrightarrow Y$, then $P$ is a component of $D_Y^{=1}$, hence a component of $\Lambda_Y ^{=1}$. Finally, since each MMP is an isomorphism over the general point of $c$, we have $X_g\cong Y_g$. 
			
		\end{proof}
	\end{lemma}

	\begin{lemma}\label{lc centers form a bounded family}
		Let $(\Cx,\Cd')\rightarrow \Cs$ be a locally stable morphism over a smooth variety $\Cs$, fix a $\mathbb{Q}$-divisor $\Cd\leq \Cd'$ such that $K_{\Cx}+\Cd$ is $\mathbb{Q}$-Cartier. Then the set 
		$$\{V\ |\ V\text{ is a log canonical center of }(\Cx_s,\Cd_s)\text{ for some closed point }s\in\Cs\}$$
		is bounded.
		\begin{proof}
			After passing to a stratification of $\Cs$, we may assume that $(\Cx,\Supp \Cd)\rightarrow \Cs$ has a fibrewise log resolution $\xi: \Cy\rightarrow \Cx$. Define $\Cd_{\Cy}$ by $K_{\Cy}+\Cd_{\Cy}\sim_{\mathbb{Q}}\xi^*(K_{\Cx}+\Cd)$, then we have
			$$K_{\Cy_s}+\Cd_{\Cy_s}\sim_{\mathbb{Q}}\xi^*(K_{\Cx_s}+\Cd_s),$$
			for any closed point $s\in \Cs$.
			It is easy to see that every log canonical center of $(\Cx_s,\Cd_s)$ is dominated by a log canonical center of $(\Cy_s,\Cd_{\Cy_s})$. 
			
			By construction, $(\Cy,\Cd_{\Cy})$ is log smooth over $\Cs$, denote its strata by $\Cv_i,i\in I$, then $\Cv_i\rightarrow\Cs$ is smooth for all $i \in I.$
			By an easy computation of discrepancy on log smooth pairs, 
			any log canonical center of $(\Cy_s,\Cd_{\Cy_s})$ is $V_{i}|_{\Cy_s}$ for some $i\in I$. Then any log canonical center of $(\Cx_s,\Cd_s)$ is isomorphic to $\xi(V_i)|_{\Cx_s}$ for some $i\in I$, and the set of families $\xi(V_i)\rightarrow \Cs,i\in I$ parametrizes all log canonical center of $(\Cx_s,\Cd_s)$, the result follows.
		\end{proof}
	\end{lemma}
	
	\begin{lemma}\label{mulplcity of ramified cover}
		Let $f:X\rightarrow T$ be a flat morphism from a normal variety to a smooth curve $T$. Suppose $\pi:S\rightarrow T$ is a ramified cover, $Y\rightarrow X\times_T S$ is the normalization of the main component and let $f_Y:Y\rightarrow S$ be the projection.

		Fix a closed point $0_T\in T$, let $0_S\in S$ be a preimage of $0_T$ in $S$. Suppose $P$ is an irreducible component of $f^*0_T$, and $Q$ is an irreducible component of the preimage of $P$ in $Y$ such that $f_Y(Q)=0_S$. Denote the ramified index of $\pi$ along $0_S$ by $r_S$, the multiplicity of $f^*0_T$ along $P$ by $m_P$. Then the degree of the finite morphism $\pi_Q:Q\rightarrow P$ is a factor of $\mathrm{min}\{r_S!,m_P!\}$.
		
		\begin{proof}
			By assumption, we have the following diagram
			$$\xymatrix{
				Q \ar@{^{(}->}[d]    \ar[r]^{\pi_Q}        &       P \ar@{^{(}->}[d]  \\
				Y\ar[d]_{f_Y} \ar[r]^{\pi_Y}  & X \ar[d] ^f \\
				S  \ar[r]_\pi   &  T   .
			}$$
			Denote the ramified index of $\pi_Y$ along the generic point of $ Q$ by $r_Q$ and $\mathrm{mult}_{Q}f_Y^*0_S$ by $m_Q$. 
			
			Next, we calculate the multiplicity of $f_Y^*0_T$ along $Q$, by the definition of the ramified index, we have
			$$\mathrm{mult}_Q\pi_Y^*f^*0_T=m_P\mathrm{mult}_Q\pi_Y^*P=m_Pr_Q.$$
			On the other hand, since $\pi_Y\circ f=f_Y\circ \pi$, we have
			$$\mathrm{mult}_Qf_Y^*\pi^*0_T=r_S\mathrm{mult}_Qf_Y^*0_S=r_Sm_Q.$$
			
			Choose a general point $x\in P$, the degree of $\pi_Q$ is equal to the number of points in $\pi_Q^{-1}(x)$. By comparing the preimages of $x$ in $Y$ (with multiplicity), we have
			$$\mathrm{deg}(\pi_Q)r_Q\leq r_S.$$
			Multiply both side by $m_Q$, we have
			$$\mathrm{deg}(\pi_Q)r_Qm_Q\leq r_Sm_Q=r_Qm_P.$$
			Then we have $\mathrm{deg}(\pi_Q)m_Q\leq m_P$. Since $\mathrm{deg}(\pi_Q)$ is both a factor of a positive integer $\leq r_S$ and a factor of a positive integer $\leq m_P$, then $\mathrm{deg}(\pi_Q)$ is a factor of $\mathrm{min}\{r_S!,m_P!\}$.
			
		\end{proof}
		
	\end{lemma}

	\section{Semistable reduction and Toroidal embedding}
	\subsection{Toric varieties}
	We fix the following notations:
	\begin{itemize}
		\item $N\cong \mathbb{Z}^n$: a lattice;
		\item $M=Hom(N,\mathbb{Z})$: the dual lattice of $N$;
		\item $T_N=Hom(M,\mathbb{C}^*)$: the complex torus associated to $N$;
		\item $\Sigma_N$: a fan in $N_{\Sigma}$;
		\item $\tau \prec \sigma$: the relation between cones $\tau$ and $\sigma$ that $\tau$ is in the face of $\sigma$;
		\item $X_\Sigma$: the toric variety associated to $\Sigma$;
		\item $U_\sigma$: the local affine chart of $X_\Sigma$ associated to $\sigma$ in $\Sigma$;
		\item $x_\sigma \in U_\sigma$: the distinguished point associated to $\sigma$;
		\item $O_\sigma$: the $T_N$-orbit of $X_\sigma$ under the $T_N$-action on $X_\Sigma$;
		\item $N_\sigma$: the sublattice of $N$ generated as a subgroup by $\sigma\cap N$;
		\item $Span_{\mathbb{R}}(\sigma)$: the real vector subspace spanned by $\sigma$;
		\item interior of $\sigma$: the topological interior of $\sigma$ in $Span_{\mathbb{R}}(\sigma)$.
	\end{itemize}
	
	Let $\Sigma',\Sigma$ be fans in $N'$, $N$ respectively. A map between fans, in notation $\psi:\Sigma'\rightarrow \Sigma$, is a homomorphism $\psi:N'\rightarrow N$ of lattices that satisfies the condition: For each $\sigma'\in \Sigma'$, there exists a $\sigma\in \Sigma$ such that $\psi(\sigma')\subset \sigma$. Such $\psi$ determines a morphism $\tilde{\psi}:X_{\Sigma'}\rightarrow X_\Sigma$. A morphism between toric varieties that arises in this way is called a toric morphism.
	
	Let $\Sigma'_\sigma$ be the set of cones in $\Sigma'$ whose interior is mapped to the interior of $\sigma\in \Sigma$. Let $\sigma'\in \Sigma'_\sigma$. The image $\psi(N'/ N'_{\sigma'})$ in $N/N_\sigma$ is independent of the choice of $\sigma'$ in $\Sigma'_\sigma$ . We define the index $[N/N_\sigma:\psi(N'/N'_\sigma)]$ to be the index of $\tilde{\psi}$ over $O_\sigma$, and denote it by $Ind(\sigma)$.
	
	Let $\tau'\in \Sigma'_\sigma$ and $\{\sigma'_1,\sigma'_2,...\}$ be the set of cones in $\Sigma'_\sigma$ that contains $\tau'$ as a face. Then each $\sigma'_i$ determines a cone $\bar{\sigma'_i}$ in $\psi^{-1}((N_\sigma)_\mathbb{R})/(N'_{\tau'})_{\mathbb{R}}$, defined by 
	$$\bar{\sigma'_i}=(\sigma'_i+(N'_{\tau'})_{\mathbb{R}})/(N'_{\tau'})_{\mathbb{R}}.$$
	Note that $\sigma'_i+(N'_{\tau'})_{\mathbb{R}}$ is contained in $\psi^{-1}(N_{\sigma})_{\mathbb{R}}$ since $\tau',\sigma'_i \in \Sigma'_\sigma$. Thus, $\{\bar{\sigma'_1},\bar{\sigma'_2},...\}$ defines a fan in $\psi^{-1}((N_\sigma)_\mathbb{R})/(N'_{\tau'})_\mathbb{R}$.
	The fan in $\psi^{-1}((N_\sigma)_\mathbb{R})/(N'_{\tau'})_\mathbb{R}$ constructed above will be called the relative star of $\tau'$ over $\sigma$ and will be denoted by $Star_\sigma(\tau')$
	
	A cone $\tau'\in\Sigma'_\sigma$ is called primitive with respect to $\psi$ if none of the faces of $\tau'$ are in $\Sigma'_\sigma$. 
	
	Let $X_\Sigma$ be a toric variety, we call the divisor $D_{\Sigma}:=X_\Sigma\setminus T$ the toric boundary of $X_\Sigma$.
	\begin{thm}{\cite[Proposition 2.1.4]{HLY02}}\label{fibration of toric morphism}
		Let $\tilde{\psi}:X_{\Sigma'}\rightarrow X_{\Sigma}$ be a toric morphism induced by a map of fans $\psi:\Sigma'\rightarrow \Sigma$. Then:
		\begin{itemize}
			\item The image $\tilde{\psi}(X_{\Sigma'})$ of $\tilde{\psi}$ is a subvariety of $X_\Sigma$. It is realized as the toric variety corresponding to the fan $\Sigma_\psi:=\Sigma \cap \psi(N'_{\mathbb{R}})$.
			
			\item The fiber of $\tilde{\psi}$ over a point $y\in X_{\Sigma_{\psi}}$ depends only on the orbit $O_\sigma,\sigma\in \Sigma_\psi$, that contains $y$. Denote this fiber by $F_\sigma$, then it can be described as follows.
			
			Define $\Sigma'_\sigma$ to be the set of cones $\sigma'$ in $\Sigma'$, whose interior is mapped to the interior of $\sigma$. Let $Ind(\sigma)$ be the index of $\tilde{\psi}$ over $O_\sigma$. Then $\psi^{-1}(y)=F_\sigma$ is a disjoint union of $Ind(\sigma)$ identical copies of connected reducible toric variety $F^c_\sigma$, whose irreducible components $F^{\tau'}_\sigma$ are the toric variety associated to the relative star $Star_\sigma(\tau')$ of the primitive elements $\tau'$ in $\Sigma'_\sigma$.
			
			\item For $\sigma\in \Sigma_\psi, \tilde{\psi}^{-1}(O_\sigma)=\tilde{O}_\sigma\times F^c_\sigma$, where $\tilde{O}_\sigma$ is a connected covering space of $O_\sigma$ of order $Ind(\sigma)$.
		\end{itemize}
		\begin{rem}
			Here the term reducible toric variety means a reducible variety obtained by gluing a collection of toric varieties along some isomorphic toric orbits.
		\end{rem}
	\end{thm}
	
	\begin{thm}{\cite[Remark 2.1.12]{HLY02}}\label{Index of toric morphism equal to 1}
		If $\psi$ is surjective, then $Ind(\sigma)=1$ for all $\sigma\in \Sigma$.
	\end{thm}

	For any toric variety $X_\Sigma$, it is well known that there is a refinement $\psi:\Sigma'\rightarrow \Sigma$, i.e. each cone of $\Sigma$ is a union of cones in $\Sigma'$. So that $\tilde{\psi}:X_{\Sigma'}\rightarrow X_\Sigma$ is a resolution of singularities.
	\begin{thm}\label{exceptional divisor of toric variety is projective bundle}
		Let $\tilde{\psi}:X_{\Sigma'}\rightarrow X_\Sigma$ be the resolution defined as above. Suppose $V$ is a divisor of $X_{\Sigma'}\setminus T_{N}$, then $V$ is birationally equivalent to $\mathbb{P}^r\times \tilde{\psi}(V)$, where $r=\mathrm{dim}V-\mathrm{dim}\tilde{\psi}(V)$.
		\begin{proof}
			Because $\tilde{\psi}$ is a toric morphism, every divisor of $X_{\Sigma'}\setminus T_{N}$ corresponds to a 1-dimension cone of $\Sigma'$. Fix a cone $\sigma\in \Sigma$ of dimension $\geq 2$. Suppose $\sigma'_1,\sigma'_2,...\in \Sigma'_\sigma$ are the 1-dimensional cones that map to the interior of $\sigma$, which are clearly primitive. Let $\sigma^!_1,\sigma^!_2,...\in \Sigma'_\sigma$ be other primitive cones. By Theorem \ref{Index of toric morphism equal to 1} and Theorem \ref{fibration of toric morphism}, $\tilde{\psi}^{-1}(O_\sigma)=O_\sigma\times F^c_\sigma$, and the irreducible components of $F^c_\sigma$ correspond to the cones $\{\sigma'_1,\sigma'_2,...\}\cup \{\sigma^!_1,\sigma^!_2,...\}$. 
			
			By comparing the dimension of exceptional locus, it is easy to see that the codimension 1 components of $\tilde{\psi}^{-1}(O_\sigma)$ equal to $O_\sigma\times F'^c_\sigma$, where the irreducible components of $F'^c_\sigma$ are the toric variety associated to the relative stars $\{Star_\sigma(\sigma'_1),Star_\sigma(\sigma'_2),...\}$. Suppose $V$ is the divisor defined by $\sigma'_1$, then $V\subset \tilde{\psi}^{-1}(O_\sigma)$ is a codimension 1 component and birational equivalent to $O_\sigma\times F^{c1}_\sigma$, where $F^{c1}_\sigma$ is the toric variety associated to the relative star $Star_\sigma(\sigma'_1)$. Because every toric variety is birational equivalent to $\mathbb{P}^r$ for some $r\in\mathbb{N}$, the result follows.
			
		\end{proof}
	\end{thm}

	\subsection{Toroidal embedding}
	Given a normal variety $X$ and an open subset $U_X\subset X$, the embedding $U_X\subset X$ is called toroidal if for every closed point $x\in X$, there exist a toric variety $X_\sigma$, a point $s\in X_\sigma$, and an isomorphism of complete local $k$-algebras
	$$\hat{\Co}_{X,x}\cong \hat{\Co}_{X_\sigma,s},$$
	such that the ideal of $X\setminus U_X$ maps isomorphically to the ideal of $X_\sigma\setminus T_\sigma$.
	In this paper we will assume that every irreducible component of $X\setminus U_X$ is normal, that is $U_X\subset X$ a strict toroidal embedding.
	
	\begin{prop}[{\cite[Page 195]{KKMS73}}]\label{etale local model}
		Let $U\subset X$ be a toroidal embedding of varieties and $x$ a closed point of $X$. Then there exists an affine toric variety $X_\sigma$ and an \'etale morphism $\psi$ from an open neighborhood of $x\in X$ to $X_\sigma$, such that locally at $x$ (for the Zariski topology) we have $U=\psi^{-1}(T)$, where $T$ is the big torus of $X_\sigma$.
	\end{prop}
	
	A dominant morphism $f:(U_X\subset X)\rightarrow (U_B\subset B)$ of toroidal embedding is called toroidal if for every closed point $x\in X$ there exist local models $(X_\sigma,s)$ at $x$, $(X_\tau,t)$ at $f(x)$ and a toric morphism $g:X_\sigma \rightarrow X_\tau$ so that the following diagram commutes
	$$\xymatrix{
		\hat{\Co}_{X,x} \ar[r]^{\cong} & \hat{\Co}_{X_\sigma,s} \\
		\hat{\Co}_{B,f(x)} \ar[r]^{\cong} \ar[u]^{\hat{f}^\# }& \hat{\Co}_{X_\tau,t} \ar[u]^{\hat{g}^\#}
	}$$
	where $\hat{f}^\#$ and $\hat{g}^\#$ are the algebra homomorphisms induced by $f$ and $g$.
	\begin{cor}[{\cite[Corollary 1.6]{AK00}}]\label{composition of toroidal is toroidal}
		If $f:(U_X\subset X)\rightarrow (U_Y\subset Y)$ and $g:(U_Y\subset Y)\rightarrow (U_Z\subset Z)$ are toroidal morphism, then $g\circ f:(U_X\subset X)\rightarrow (U_Z\subset Z)$ is toroidal.
	\end{cor}

	\begin{defn}[{\cite[Definition 2.2]{ALT19}}]
		Let $f:(X,D)\rightarrow (Z,B)$ be a projective morphism between projective normal log pairs with connected fibers, we say $f$ is semistable if
		\begin{itemize}
			\item the varieties $X$ and $Z$ admit toroidal structures $U_X:=X\setminus D\subset X$ and $U_Z:=Z\setminus B\subset Z$,
			\item with this structure, the morphism $f$ is toroidal,
			\item the morphism $f$ is equidimensional,
			\item all the fibers of the morphism $f$ are reduced, and
			\item $X$ and $Z$ are nonsingular.
		\end{itemize}
	\end{defn}
	\begin{thm}[Semistable Reduction]\label{Semistable reduction}
		Let $X\rightarrow Z$ be a projective morphism between projective normal varieties and $D\subset X$ be a closed subset. Then there exists a proper, surjective, generically finite morphism of irreducible varieties $b:Z'\rightarrow Z$, a projective birational morphism of irreducible varieties $a:X'\rightarrow (X\times_Z Z')^m$, where $(X\times_Z Z')^m$ is the main component of the fiber product $X\times_Z Z'$, and divisors $B'\subset Z'$, $D'\subset X'$, such that
		\begin{itemize}
			\item $a^{-1}(D\times _Z Z')\cup f'^{-1}(B')\subset D'$, and
			\item the morphism $f':(X',D')\rightarrow (Z',B')$ is semistable. 
		\end{itemize}
		\begin{proof}
			This is a direct result of {\cite[Theorem 4.7]{ALT19}}.
		\end{proof}
	\end{thm}
	\begin{lemma}[{\cite[Lemma 6.2]{AK00}}]\label{base change of a toroidal morphism is a toroidal morphism}
		Let $f:(X,D)\rightarrow (Z,B)$ be a semistable morphism. Let $g:C\rightarrow Z$ be a morphism such that $C$ is nonsingular and $g^{-1}(B)$ is a normal crossing divisor. Let $X_C=C\times_Z X$ and $g_C:X_C\rightarrow X,f_C:X_C\rightarrow C$ the two projections.
		
		Denote $B_C=g^{-1}(B)$ and $D_C=g_C^{-1}(D)$. Then $(U_{C}:=C\setminus B_C \subset C)$ and $(U_{X_C}:=X_C\setminus D_C \subset X_C)$ are toroidal embeddings, and $f_C:(U_{X_C}\subset X_C)\rightarrow (U_C\subset C)$ is an equidimensional toroidal morphism with reduced fibers.
	\end{lemma}
	
	\begin{lemma}\label{log canonical center is projective bundle}
		Let $X$ be a projective normal variety, $D$ a reduced divisor on $X$, and $U_X:=X\setminus D\subset X$ a toroidal embedding. Suppose $\Delta\leq D$ is a $\mathbb{Q}$-divisor such that $(X,\Delta)$ is sub-log canonical. 
		
		If $P$ is a log canonical place of $(X,\Delta)$, then $P$ is birational equivalent to $\mathbb{P}^r\times V$, where $V$ is the image of $P$ in $X$ and $r=\mathrm{dim}X-\mathrm{dim}V-1$.
		\begin{proof}
			Let $P$ be a log canonical place of $(X,\Delta)$, suppose $x$ is a general point of the image of $P$ on $X$. For the rest of the proof, we consider Zariski locally near $x$ by replacing $X$ with an open neighborhood of $x$. 
			
			Let $x$ be a general point of $V\subset X$ and $X_\sigma$ the affine toric variety defined in Proposition \ref{etale local model}. Let $\sigma\subset \sigma'$ be a subdivision such that $X_{\sigma'}\rightarrow X_\sigma$ is a resolution. Because $\pi$ is \'etale, $X_1:=X_{\sigma'}\times_{X_\sigma}X$ is a log resolution of $(X,D)$. We have the following diagram
			$$\xymatrix{
				X_1\ar[d]_{h} \ar[r]^{\pi_1} & X_{\sigma_1} \ar[d]^{h_{\sigma}} \\
				X  \ar[r]_{\pi}            & X_{\sigma}
			}$$
			
			Let $D_1$ be the strict transform of $D$ on $X_1$ plus the $h$-exceptional divisor, then $h:(U_1:=X_1\setminus D_1\subset X_1)\rightarrow (U_X\subset X)$ is a toroidal morphism. 
			By an easy computation of discrepancies on snc divisors, it is easy to see that $P$ can be obtained by a sequence of blows up along strata of $(X',D')$. We will show that such morphism is \'etale locally equal to a toric morphism between toric varieties.

			Suppose we have a sequence of blows up $h_i:X_{i+1}\rightarrow X_i,1\leq i\leq k-1$ along a strata $V_i$ of $(X_i,D_i)$, where $D_{i+1}$ is the strict transform of $D_i$ plus the $h_i$-exceptional divisor, so that $P$ is a divisor on $X_k$. Next, we show that there is a Cartesian diagram
			$$\xymatrix{
				X_j \ar[d]_{g_j}\ar[r]^{\pi_j} & X_{\sigma_j} \ar[d]\\
				X_1 \ar[r]^{\pi_1}    & X_{\sigma_1}
			}$$
			where 
			\begin{itemize}
				\item the horizontal arrows are \'etale morphisms, 
				\item $\sigma_j$ is a subdivision of $\sigma_1$,
				\item $X_{\sigma_j}\rightarrow X_{\sigma_1}$ is the corresponding toric morphism, and
				\item near any closed point of $g_j^{-1}x_1$, we have $U_{j}=\pi_{j}^{-1}T_{j}$, where $T_{j}$ is the big torus of $X_{\sigma_{j}}$,
			\end{itemize}
			for all $1\leq j\leq k$.
			
			Suppose it is true for $j=i$. 
			Let $X_{\sigma_{i+1}}\rightarrow X_{\sigma_i}$ be the toric morphism determined by blowing up $X_{i}$ along the image of $V_i$ on $X_{\sigma_i}$.
			Because blowing up is uniquely determined by local equations and both $X_{i+1}\rightarrow X_i$ and $X_{\sigma_{i+1}}\rightarrow X_{\sigma_i}$ are obtained by blowing up the same subvariety \'etale locally, then there is a natural \'etale morphism $\pi_{i+1}:X_{i+1}\rightarrow X_{\sigma_{i+1}}$ such that near any closed point of $g_j^{-1}x$, we have $U_{i+1}=\pi_{i+1}^{-1}T_{i+1}$, where $T_{i+1}$ is the big torus of $X_{\sigma_{i+1}}$. Because the composition of $X_{\sigma_{i+1}}\rightarrow X_{\sigma_i}$ and $X_{\sigma_i}\rightarrow X_{\sigma_1}$ is a toric morphism, the claim is true for $j=i+1$.
			
			Now we have the following Cartesian diagram 
			$$\xymatrix{
				X_k\ar[d]_{f} \ar[r]^{\pi_k} & X_{\sigma_k} \ar[d]^{f_{\sigma}} \\
				X  \ar[r]_{\pi}            & X_{\sigma}.
			}$$
			By assumption, $P$ is a divisor on $X_k$ and $\pi_k$ is \'etale near the general point of $P$. Then $P$ is equal to the pull back of a divisor $P_{\sigma_k}$ on $X_{\sigma_k}$. Because $\sigma_k\rightarrow \sigma$ is a subdivision, by Lemma \ref{exceptional divisor of toric variety is projective bundle}, $f_\sigma|_{P_{\sigma_k}}$ is birationally equivalent to a $\mathbb{P}^r$-bundle. Because the diagram is Cartesian, $f|_P$ is also birationally equivalent to a $\mathbb{P}^r$-bundle. Therefore, $P$ is birationally equivalent to $f(P)\times \mathbb{P}^r$. 
			
		\end{proof}
	\end{lemma}

	\section{Moduli of polarised Calabi-Yau pairs}
	In this section we recall some definitions and results on moduli of stable pairs and polarised log Calabi-Yau pairs, see \cite{Kol23}, \cite{KX20}, \cite{Bir22}, and \cite{Bir23}. We fix natural numbers $d,n$ and positive rational numbers $c,v$.
	
	\begin{defn}\label{definition of log Calabi-Yau pairs}
		A log Calabi-Yau pair is a semi-log canonical pair $(X,\Delta)$ such that $K_X+\Delta\sim_{\mathbb{Q}} 0$.
		
		A polarised log Calabi-Yau pair consists of a log Calabi-Yau pair $(X,\Delta)$ and an ample integral divisor $N\geq 0$ such that $(X,\Delta+uN)$ is semi-log canonical for some real number $u>0$. Fix a natural number $d$ and positive rational numbers $c,v$.
		
		A $(d,c,v)$-polarised log Calabi-Yau pair is a polarised log Calabi-Yau pair $(X,\Delta),N$ such that $\mathrm{dim}X=d,\Delta=cD$ for some integral divisor $D$, and $\vol(N)=v$.
	\end{defn}
	\begin{defn}
		Let $S$ be a reduced scheme. A $(d,c,v)$-polarised Calabi-Yau family over $S$ consists of a projective morphism $f:X\rightarrow S$ of schemes, a $\mathbb{Q}$-divisor $B$, and an integral divisor $N$ on $X$ such that
		\begin{itemize}
			\item $(X,B+uN)\rightarrow S$ is a stable family for some rational number $u>0$ with fibers of pure dimension $d$,
			\item $B=cD$ where $D\geq 0$ is a relative Mumford divisor,
			\item $N\geq 0$ is a relative Mumford divisor,
			\item $K_{X/S}+B\sim_{\mathbb{Q}}0/S$, and
			\item for any fiber $X_s$ of $f$, $\mathrm{vol}(N|_{X_s})=v$.
		\end{itemize}
	\end{defn}
	\begin{rem}
		The definition of $(d,c,v)$-polarised log Calabi-Yau pair and $(d,c,v)$-polarised Calabi-Yau families comes from the Chapter 7 in the first arxiv version of \cite{Bir23}.
	\end{rem}
	\begin{lemma}\label{lc threshold is bounded from below}
		There exist a positive rational number $t$ and a natural number $r$ such that $rc,rt\in \mathbb{N}$ satisfying the following. Assume $(X,B),N$ is a $(d,c,v)$-polarised slc Calabi-Yau pair over a field of characteristic zero. Then
		\begin{itemize}
			\item $(X,B+tN)$ is slc,
			\item $B+tN$ uniquely determines $B,N$, and
			\item $r(K_X+B+tN)$ is very ample with
			$$h^j(mr(K_X+B+tN))=0$$
			for $m,j>0$.
		\end{itemize}
		\begin{proof}
			This is Lemma 7.7 in the first arxiv version of \cite{Bir23}.
		\end{proof}
	\end{lemma}
	
	The following definition comes from Chapter 7 in the first arxiv version of \cite{Bir23}. 
	
	Let $t$ be as in the lemma. To simplify notation, let $\Theta=(d,c,v,t,r,\mathbb{P}^n)$. Let $S$ be a reduced scheme. A strongly embedded $\Theta$-polarised Calabi-Yau family over $S$ is a $(d,c,v)$-polarised Calabi-Yau family $f:(X,B),N\rightarrow S$ together with a closed embedding $g:X\rightarrow \mathbb{P}^n_S$ such that
	\begin{itemize}
		\item $(X,B+tN)\rightarrow S$ is a stable family,
		\item $f=\pi g$ where $\pi$ denotes the projection $\mathbb{P}^n_S\rightarrow S$,
		\item letting $\Cl:=g^*\Co_{\mathbb{P}^n_S}(1)$, we have $R^qf_*\Cl\cong R^q\pi_*\Co_{\mathbb{P}^n_S}(1)$ for all $q$, and
		\item for every $s\in S$, we have
		$$\Cl_s\cong \Co_{X_s}(r(K_{X_s}+B_s+tN_s)).$$
	\end{itemize}
	We denote the family by $f:(X\subset \mathbb{P}^n_S ,B),N\rightarrow S$.
	
	Define the functor $\Ce^s\mathcal{PCY}_{\Theta}$ on the category of reduced schemes by setting
	$$\Ce^s\mathcal{PCY}_{\Theta}(S) =\{\text{strongly embedded }\Theta\text{-polarised slc Calabi-Yau families over }S\}.$$
	
	By Proposition 7.8 in the first arxiv version of \cite{Bir23}, the functor $\Ce^s\mathcal{PCY}_{\Theta}$ has a fine moduli space, which is a reduced separated scheme $\Cs:=E^sPCY_{\Theta}$, and a universal family $(\Cx\subset \mathbb{P}^n_{\Cs},\Cd),\Cn\rightarrow \Cs$.

	\section{Proof of Main Theorem}
	
	\begin{lemma}\label{log bounded}
		Fix a natural number $d$ and a positive rational number $c$. Suppose $(X,\Delta)$ is an $\epsilon$-lc of dimension $d$, $-(K_X+\Delta)$ is ample and $\mathrm{coeff}\Delta\geq  c$. Then $(X,\Delta)$ is log bounded.
		\begin{proof}
			By the main theorem of \cite{Bir21}, $X$ is bounded. Then there exist a natural number $n$, two constants $V_1,V_2$ depending only on $d$ and $\epsilon$, and a very ample divisor $H$ on $X$ defining an embedding $X\subset \mathbb{P}^n$ such that $H^d\leq V_1$ and $H^{d-1}\cdot K_{X}\geq -V_2$. Because $\mathrm{coeff}\Delta\geq c$, we have
			\begin{equation*}
				\begin{aligned}
					cH^{d-1}\cdot\Supp \Delta & \leq H^{d-1}\cdot \Delta\\
					& = H^{d-1}\cdot(K_{X}+\Delta)-H^{d-1}\cdot K_{X} \\
					& \leq -H^{d-1}\cdot K_{X}\\
					& \leq V_2
				\end{aligned}
			\end{equation*}
			By the boundedness of the Chow variety, both $X$ and $\Supp\Delta$ are parametrized by a subvariety of the Hilbert scheme. Then $(X,\Delta)$ is log bounded.
		\end{proof}
	\end{lemma}
	
	\begin{lemma}\label{1.3 implies 1.2}
		Theorem \ref{Main theorem, flat morphism}(2) implies Theorem \ref{Main theorem, flat morphism}(1) and Theorem \ref{Main theorem, locally stable morphism}.
		
		\begin{proof}
			Let $f:(X,\Delta)\rightarrow C$ be a fibration over a curve such that 
			\begin{itemize}
				\item $\mathrm{coeff}\Delta\subset c\mathbb{N}$,
				\item $-(K_X+\Delta)$ is nef and big over $C$, and
				\item the general fiber $(X_g,\Delta_g)$ of $f$ is $\epsilon$-lc Fano.
			\end{itemize}
			Fix a closed point $0\in C$, for any log canonical place $P$ of $(X,\Delta+\mathrm{lct}(X,\Delta;f^*0)f^*0)$, by Lemma \ref{relative complement}, there exists a birational map $Y\dashrightarrow X$ and a log pair $(Y,\Lambda_Y)$, such that
			\begin{itemize}
				\item $P$ is a log canonical place of $(Y,\Lambda_Y)$, 
				\item $(Y,\Lambda_Y)$ is log canonical over a neighborhood of $s$, and
				\item the general fiber $Y_g$ is isomorphic to $X_g$.
			\end{itemize}
			
			Let $\Delta_Y$ be the strict transform of $\Delta$ on $Y$. Because the general fiber $(X_g,\Delta_g)$ is $\epsilon$-lc, $-(K_{X_g}+\Delta_g)$ is ample and $\mathrm{coeff}\Delta_g\geq c$, by Lemma \ref{log bounded}, $(X_g,\Delta_g)$ is log bounded, then $(Y_g,\Delta_{Y_g})$ is also log bounded. 
			
			By log boundedness, there exists a natural number $m$ and an open subset $U\subset C$ such that $-m(K_{Y_u}+\Delta_{Y,u})$ is very ample without higher cohomology for any $u\in U$. Choose a general member $N\in |-m(K_Y+\Delta_Y)|_U$. By very ampleness, $(Y_g,\Lambda_{Y,g}+tN_g)$ is log canonical for a positive rational number $t$. Hence $(Y_g,\Lambda_{Y,g}),N_g$ is polarised Calabi-Yau pair. We have transferred the condition of Theorem \ref{Main theorem, flat morphism}(1) to (2).
			
			Next, we show that Theorem \ref{Main theorem, flat morphism} implies Theorem \ref{Main theorem, locally stable morphism}. 
			
			Suppose $f:(X,\Delta)\rightarrow C$ is a locally stable morphism. By definition, $X_s=f^*(s)$ is reduced and $(X,\Delta+f^*(s))$ is log canonical for every closed point $s\in C$. Then every irreducible component of $X_s$ is a log canonical place of $(X,\Delta+f^*(s))$. 
			
			Suppose Theorem \ref{Main theorem, flat morphism} is true. Because $\mathrm{mult}_Pf^*(s)=1$ for every irreducible component $P\subset X_s$, there exists a bounded family $\Cw\rightarrow \Ct$ and a finite dominant rational map $\Cw_t\dashrightarrow P$ whose degree is a factor of $\mathrm{min}\{l,\mathrm{mult}_Pf^*(s)\}=1$. Then $\Cw_t\dashrightarrow P$ is a birational map, which means $P$ is birationally bounded.
		\end{proof}
	\end{lemma}
	
	\begin{proof}
		By Lemma \ref{1.3 implies 1.2}, we only need to prove Theorem \ref{Main theorem, flat morphism}(2).
		
		Suppose $(X,\Delta)\rightarrow C$ is a fibration, $N$ is a divisor on $X$ such that
		\begin{itemize}
			\item $K_X+\Delta\sim_{\mathbb{Q},C}0$, and
			\item the general fiber $(X_g,\Delta_g),N_g$ is a $(d,c,v)$-polarised Calabi-Yau pair.
		\end{itemize}
		By Lemma \ref{lc threshold is bounded from below}, there exists $t\in \mathbb{Q}^{\geq 0}$ and $r\in\mathbb{ N}$ such that $(X_g,\Delta_g+tN_g)$ is slc and $r(K_{X_g}+\Delta_g+tN_g)$ is very ample without higher cohomology. Then by cohomology and base change, $r(K_{X_U}+\Delta_U+tN_U)$ is very ample over an open subset $U\subset C$, and it defines a closed embedding $g:X_U\hookrightarrow \mathbb{P}^n_U$. 
		Also because $(X_U,\Delta_U+tN_U)\rightarrow U$ is a stable family, $f_U:(X_U\subset \mathbb{P}^n_U,\Delta_U),N_U\rightarrow U$ is a strongly embedded polarised slc Calabi-Yau family over $U$. Since $\Ce^s\mathcal{PCY}_{\Theta}$ has a fine moduli space with the universal family $(\Cx\subset \mathbb{P}^n_{\Cs},\Cd),\Cn\rightarrow \Cs$, we have $(X_U,\Delta_U)\cong (\Cx,\Cd)\times_\Cs U$, where $U\rightarrow \Cs$ is the moduli map defined by $f_U$.
		
		After passing to a stratification of $\Cs$, we may assume that there is a fibrewise log resolution $\xi:(\Cy,\Cd_\Cy)\rightarrow (\Cx,\Cd)$, where $\Cd_\Cy$ is defined by $K_{\Cy}+\Cd_\Cy=\xi^*(K_{\Cx}+\Cd)$. By Theorem \ref{Semistable reduction}, there is a generically finite morphism $\bar{\Cs}^o\rightarrow \Cs$ and a closure $\bar{\Cs}^o\hookrightarrow \bar{\Cs}$ such that the pull back $(\Cy,\Cd_{\Cy})\times_{\Cs}\bar{\Cs}^o$ extend to a semistable morphism 
		$$\chi:(\bar{\Cy},\bar{\Cd}'_{\bar{\Cy}})\rightarrow (\bar{\Cs},\bar{\Cb}),$$
		where $\bar{\Cb}\supset \bar{\Cs}\setminus \bar{\Cs}^o$ and $\bar{\Cd}'_{\bar{\Cy}}\supset \chi^{-1}\bar{\Cb}\cup (\Cd_{\Cy}\times_{\Cs}\bar{\Cs}^o)$. 
		Let $\bar{\Cd}''_{\bar{\Cy}}$ be the closure of $\Cd_{\Cy}\times_{\Cs}\bar{\Cs}^o$, then $\Supp \bar{\Cd}''_{\bar{\Cy}}\subset \bar{\Cd}'_{\bar{\Cy},h}$, and by the property of weakly semistable morphism, $(\bar{\Cy},\Supp\bar{\Cd}''_{\bar{\Cy}})\rightarrow \bar{\Cs}$ is locally stable. 
		
		It is easy to see that there is an effective $\mathbb{Q}$-divisor $\bar{\Cd}^!_{\bar{\Cy}}$ which is supported on $\chi^{-1}(\bar{\Cs}\setminus \bar{\Cs}^o)$ but does not contain the whole fiber over any generic point of $\bar{\Cs}\setminus \bar{\Cs}^o$, such that $K_{\bar{\Cy}}+\bar{\Cd}''_{\bar{\Cy}}\sim_{\mathbb{Q},\bar{\Cs}} \bar{\Cd}^!_{\bar{\Cy}}$. Define $\bar{\Cd}_{\bar{\Cy}}:=\bar{\Cd}''_{\bar{\Cy}}-\bar{\Cd}^!_{\bar{\Cy}}$, then we have 
		$$K_{\bar{\Cy}}+\bar{\Cd}_{\bar{\Cy}}\sim_{\mathbb{Q},\bar{\Cs}} 0$$
		
		Let $\bar{C}$ be the closure of $U\times _{\Cs}\bar{\Cs}^o$. Then there is a finite morphism $\pi:\bar{C}\rightarrow C$, and we choose $\bar{0}$ to be a closed point that maps to $0$. Because $\bar{\Cs}$ is proper, there is a morphism $\bar{C}\rightarrow \bar{\Cs}$. Define $(\bar{Y},\bar{D}_{\bar{Y}}):=(\bar{\Cy},\bar{\Cd}_{\bar{\Cy}})\times_{\bar{\Cs}}\bar{C}$ and $\bar{D}'_{\bar{Y}}$ be the pull back of $\bar{\Cd}'_{\bar{\Cy}}$ on $\bar{Y}$. Then by Lemma \ref{base change of a toroidal morphism is a toroidal morphism}, $(\bar{Y},\bar{D}'_{\bar{Y}})\rightarrow (\bar{C},\bar{0})$ is a toroidal morphism, hence $(\bar{Y},\bar{D}'_{\bar{Y},h})\rightarrow \bar{C}$ is a locally stable morphism. In particular, $\bar{Y}\setminus \bar{D}'_Y\subset \bar{Y}$ is a toroidal embedding.
		Define $\bar{Y}_{\bar{0}}:=\bar{f}'^*\bar{0}=\mathrm{red}(\bar{f}'^{-1}\bar{0})$, because $(\bar{Y},\bar{D}'_{\bar{Y},h})\rightarrow \bar{C}$ is locally stable and $\bar{D}_{\bar{Y}}\leq \bar{D}'_{\bar{Y},h}$, then $(\bar{Y},\bar{D}_{\bar{Y}}+\bar{Y}_{\bar{0}})$ is sub-log canonical and it is easy to see that $$K_{\bar{Y}}+\bar{D}_{\bar{Y}}+\bar{Y}_{\bar{0}}\sim_{\mathbb{Q},\bar{C}}0.$$
		
		Let $\bar{X}$ be the normalization of $X\times_C \bar{C}$ and denote the natural morphism $\bar{X}\rightarrow \bar{C}$ by $\pi_X$. Write $\alpha:=\mathrm{lct}(X,\Delta;f^*0)$. By the Hurwitz's formula, there is a $\mathbb{Q}$-divisor $\bar{\Delta}_\alpha$ such that
		$$K_{\bar{X}}+\bar{\Delta}_\alpha\sim_{\mathbb{Q}}\pi_X^*(K_{X}+\Delta+\alpha f^*0).$$
		Suppose $P$ is a log canonical place of $(X,\Delta+\alpha f^*0)$. Let $X'\rightarrow X$ be a dlt modification of $(X,\Delta+\alpha f^*0)$ such that $P$ is a divisor on $X'$, $\bar{X}'$ be the normalization of $X'\times_C\bar{C}$, and $\bar{P}$ be an irreducible component of the preimage of $P$ on $\bar{X}'$. By \cite[2.41]{Kol13}, $\bar{P}$ is a log canonical center of $(\bar{X},\bar{\Delta}_\alpha)$. And by Lemma \ref{mulplcity of ramified cover}, the degree of the finite morphism $\bar{P}\rightarrow P$ is a factor of $\mathrm{min}\{\mathrm{deg}(\pi)!,\mathrm{mult}_Pf^*0\}$. Define $l$ to be the factorial of the degree of the finite morphism $\bar{\Cs}\rightarrow \Cs$, which is clearly greater or equal to $\mathrm{deg}(\pi)!$, then $\mathrm{min}\{\mathrm{deg}(\pi)!,\mathrm{mult}_Pf^*0\}$ is a factor of $\mathrm{min}\{l,\mathrm{mult}_Pf^*0\}$. Thus we only need to prove that $\bar{P}$ is birationally bounded.
		
		Because $K_{\bar{X}}+\bar{\Delta}_\alpha\sim_{\mathbb{Q},\bar{C}}0,K_{\bar{Y}}+\bar{D}_{\bar{Y}}+\bar{Y}_{\bar{0}}\sim_{\mathbb{Q},\bar{C}}0$, the general fiber of $(\bar{X},\bar{\Delta}_\alpha)\rightarrow \bar{C}$ is crepant birationally equivalent to the general fiber of $(\bar{Y},\bar{D}_{\bar{Y}}+\bar{Y}_{\bar{0}})\rightarrow \bar{C}$,
		and both $(\bar{X},\bar{\Delta}_\alpha)$ and $(\bar{Y},\bar{D}_{\bar{Y}}+\bar{Y}_{\bar{0}})$ have a log canonical place that dominates $\bar{0}$, $(\bar{X},\bar{\Delta}_\alpha)$ is crepant birationally equivalent to $(\bar{Y},\bar{D}_{\bar{Y}}+\bar{Y}_{\bar{0}})$. In particular, a divisor $\bar{P}$ is a log canonical place of $(\bar{X},\bar{\Delta}_\alpha)$ if and only if it is a log canonical place of $(\bar{Y},\bar{D}_{\bar{Y}}+\bar{Y}_{\bar{0}})$.
		
		Recall that $\bar{Y}\setminus \bar{D}'_Y\subset \bar{Y}$ is a toroidal embedding and $\Supp(\bar{D}_{\bar{Y}}+\bar{Y}_{\bar{0}})\leq \bar{D}'_Y$, by Lemma \ref{log canonical center is projective bundle}, $\bar{P}$ is birationally equivalent to $\bar{V}\times \mathbb{P}^r$, where $\bar{V}$ is the image of $\bar{P}$ on $\bar{Y}$. It is easy to see that $\bar{V}$ is a log canonical center of $(\bar{Y},\bar{D}_{\bar{Y}}+\bar{Y}_{\bar{0}})$. By applying Lemma \ref{toric singularities} to a local model of $(\bar{Y}\setminus \Supp \bar{D}'_{\bar{Y}}\subset \bar{Y})$ near a general point of $\bar{V}$, we can see that there is an irreducible component $\bar{Y}_{\bar{0}}^1$ of $(\bar{D}_{\bar{Y}}+\bar{Y}_{\bar{0}})^{=1}$, which maps to $\bar{0}$, such that $\bar{V}$ is a subvariety of $\bar{Y}_{\bar{0}}^1$. 
		
		To prove $\bar{P}$ is birationally bounded, we only need to prove all log canonical centers of $(\bar{Y},\bar{D}_{\bar{Y}}+\bar{Y}_{\bar{0}})$ are in a bounded family.
		
		If $\bar{P}$ has codimension 1 in $\bar{Y}$, then $\bar{P}$ is just $\bar{Y}^1_{\bar{0}}$. Since $\bar{Y}_{\bar{0}}$ is in a bounded family $\bar{\Cy}\rightarrow \bar{\Cs}$ and $\bar{Y}^1_{\bar{0}}$ is an irreducible component of $\bar{Y}_{\bar{0}}$, $P$ is birationally bounded. 
		
		If $\bar{P}$ has codimension $\geq 2$ in $\bar{Y}$. Apply adjunction on $(\bar{Y},\bar{D}_{\bar{Y}}+\bar{Y}_{\bar{0}})$, we have
		$$(K_{\bar{Y}}+\bar{D}_{\bar{Y}}+\bar{Y}_{\bar{0}})|_{\bar{Y}^1_{\bar{0}}}=K_{\bar{Y}^1_{\bar{0}}}+\bar{D}^1_{\bar{Y},\bar{0}}.$$
		By inverse of adjunction, a log canonical center of $(\bar{Y},\bar{D}_{\bar{Y}}+\bar{Y}_{\bar{0}})$ intersecting $\bar{Y}^1_{\bar{0}}$ corresponds to a log canonical center of $(\bar{Y}^1_{\bar{0}},\bar{D}^1_{\bar{Y},\bar{0}})$, hence also a log canonical center of $(\bar{Y}_{\bar{0}},\bar{D}'_{\bar{Y},\bar{0}})$. Let $s\in\bar{\Cs}$ be the image of $\bar{C}$ in $\bar{\Cs}$, by the definition of $\bar{Y}$, we have the isomorphism 
		$$(\bar{Y}_{\bar{0}},\bar{D}'_{\bar{Y},\bar{0}})\cong (\bar{\Cy}_s,\bar{\Cd}'_{\bar{\Cy}_s}).$$
		By Lemma \ref{lc centers form a bounded family}, the log canonical centers of $(\bar{\Cy}_s,\bar{\Cd}'_{\bar{\Cy}_s})$ is in a bounded family, then all log canonical centers of $(\bar{Y},\bar{D}_{\bar{Y}}+\bar{Y}_{\bar{0}})$ is in a bounded family.

	\end{proof}
	
	\begin{rem}
		In Theorem \ref{Main theorem, flat morphism}(1), if we further assume $(X,\Delta)$ to be klt, we will show that the condition $-(K_X+\Delta)$ being ample over $C$ can be replaced by $-(K_X+\Delta)$ being nef and big over $C$. 
		
		Assume $-(K_X+\Delta)$ is nef and big over $C$, because $(X,\Delta)$ is klt, by \cite[Corollary 1.3.2]{BCHM10}, $-(K_X+\Delta)$ is semiample, hence defines a contraction $h:X\rightarrow Y$. Let $\Delta_Y:=h_*\Delta$, then we have
		\begin{itemize}
			\item $K_X+\Delta\sim_{\mathbb{Q},c}h^*(K_Y+\Delta_Y)$,
			\item the general fiber $(Y_g,\Delta_{Y_g})$ is $\epsilon$-lc,
			\item $-(K_Y+\Delta_Y)$ is ample over $C$, and
			\item $\mathrm{lct}(X,\Delta;f^*0)=\mathrm{lct}(Y,\Delta_Y;f_Y^*0)$, where $f_Y$ is the natural morphism $Y\rightarrow C$.
		\end{itemize}
		It is easy to see that $(X,\Delta+\mathrm{lct}(X,\Delta;f^*0)f^*0)$ is crepant birationally equivalent to $(Y,\Delta_Y+\mathrm{lct}(Y,\Delta_Y;f_Y^*0)f_Y^*0)$, and a divisor $P$ is a log canonical place of $(X,\Delta+\mathrm{lct}(X,\Delta;f^*0)f^*0)$ if and only if it is a log canonical place of $(Y,\Delta_Y+\mathrm{lct}(Y,\Delta_Y;f_Y^*0)f_Y^*0)$. Then we can replace $(X,\Delta)$ by $(Y,\Delta_Y)$ and the result follows.

		In Theorem \ref{Main theorem, locally stable morphism}(1), for the locally stable morphism $f:(X,\Delta)\rightarrow C$ over a smooth curve, because the general fiber $(X_g,\Delta_g)$ is klt, by \cite[Proposition 2.14]{Kol23}, $(X,\Delta)$ is a klt pair. Then for the same reason, the condition $-(K_X+\Delta)$ being ample over $C$ can be replaced by $-(K_X+\Delta)$ being nef and big over $C$. 
		
	\end{rem}
	
	\section{Boundedness of the boundary}
	
	In this section, we prove the following technical result. We think it is useful in the study of singularities of Fano fibrations.
	
	\begin{thm}\label{crepant birational boundedness}
		Fix a natural number $d$ and positive rational number $c,\epsilon$. Then there exist a natural number $l$ and a bounded family of log pairs $(\Cw,\Cd)\rightarrow \Ct$, such that:
		
		Let $X$ be a normal quasi-projective variety, $(X,\Delta)$ a log canonical pair, and $f:X\rightarrow C$ a fibration of relative dimension $d$ over a smooth curve $C$. Suppose
		\begin{itemize}
			\item $\mathrm{coeff}\Delta\subset c\mathbb{N}$,
			\item $-(K_X+\Delta)$ is ample over $C$, and
			\item the general fiber $(X_g,\Delta_g)$ is $\epsilon$-lc.
		\end{itemize}
		Then for any closed point $0\in C$ and any divisor $P$ over $X$ which is a log canonical place of $(X,\Delta+\mathrm{lct}(X,\Delta;f^*0)f^*0)$, there is a diagram
		$$\xymatrix{
			\bar{X} \ar[r]^{\pi_X} \ar[d]_{\bar{f}}& X \ar[d]^f\\
			\bar{C} \ar[r]_\pi & C
		}$$
		and a $\mathbb{Q}$-divisor $\bar{\Lambda}$ on $\bar{X}$ with the following properties:
		\begin{itemize}
			\item $l(K_{\bar{X}}+\bar{\Lambda})\sim_{\bar{C}}0$.
			\item $\bar{C}\rightarrow C$ is finite of degree $\leq l$.
			\item the induced morphism $\bar{X}\rightarrow X\times _C \bar{C}$ is birational.
			\item $P$ is dominated by a divisor $\bar{P}$ on $\bar{X}$, that is, if $X'\rightarrow X$ is a birational map and extracts $P$, then $P$ is dominated by $\bar{P}$ via the natural rational map $\bar{X}\dashrightarrow X'$.
			\item $\bar{P}$ is an irreducible component of $\bar{\Lambda}^{=1}$.
			\item $(\bar{X},\bar{\Lambda})$ is sub-log canonical in a neighborhood of $\bar{P}$.
			\item there is a closed point $t\in \Ct$, such that $(\Cw_t,\Cd_t)$ is crepant birationally equivalent to $(\bar{P},\mathrm{Diff}_{\bar{P}}(\bar{\Lambda}-\bar{P}))$.
		\end{itemize}
	\end{thm}
	
	To prove this theorem, we need to generalize some results in section 3 to the log pair case. First, we consider a toric variety with a boundary.
	
	\begin{lemma}\label{toric singularities}
		
		Let $(X_\Sigma,D=\sum_{i}d_iD_i)$ be a sub-pair where $X_\Sigma$ is a normal toric variety and $D_i$ is torus-invariant. It is well known that if $d_i\leq 1$ for all $i$, then $(X_\Sigma, D)$ is sub-log canonical. 
		
		Suppose $\psi:\Sigma'\rightarrow \Sigma$ is a refinement and $\tilde{\psi}:X_{\Sigma'}\rightarrow X_{\Sigma}$ be the induced birational morphism. Let $\tau'\in \Sigma'$ be a 1-dimensional cone and $\sigma\in \Sigma$ be the cone such that $\tau'\in \Sigma'_\sigma$. Suppose $\{\tau_1,...,\tau_k\}$ is the set of 1-dimensional faces of $\sigma$ and $B_1,...,B_k$ are the corresponding divisors on $X$. Let $P$ denote the divisor corresponding to $\tau' $ over $X$, define $D'$ by $K_{X_\Sigma'}+P+D'\sim_{\mathbb{Q}}\tilde{\psi}^*(K_{X_\Sigma}+D)$. 
		
		If $P$ is a log canonical place of $(X_\Sigma,D)$, then $\mathrm{coeff}_{B_i}D=1$ for $1\leq i\leq k$, and for every $\tilde{\psi}|_{P}$-horizontal prime divisor $D_P$ of $\Supp\mathrm{Diff}_{P}(D')$, we have $\mathrm{coeff}_{D_P}(\mathrm{Diff}_{P}(D'))=1$. 
		
		\begin{proof}
			Let $\{B_1,...,B_j\}$ be the set of linearly independent 1-dimensional cones in $\Sigma'_{\sigma}$ that contains $\tau'$ in their closure. Then by the proof of \cite[Proposition 11.4.24]{CLS11}, $\mathrm{coeff}_{B_i}D=1$ for $1\leq i\leq j$. Because $K_{X_\Sigma}+D\sim_{\mathbb{Q}}(d_i-1)D_i$ is $\mathbb{Q}$-Cartier, then there exists $m_\tau\in M$ such that $<m_\sigma,\tau_i>=d_i-1$ for all $1\leq i\leq l$. Since $\tau'$ is primitive, $\{B_1,...,B_j\}$ spans $Span_{\mathbb{R}}\sigma$. Because $<m_\sigma,\tau_i>=0$ for all $1\leq i\leq j$, $\mathrm{coeff}_{B_i}D=1$ for all $1\leq i\leq k$.
			
			Because $P$ is an orbit closure, then $P$ is normal and $\mathrm{Diff}_{P}(D')$ is well defined. By Theorem \ref{fibration of toric morphism}, the general fiber $P_g$ of $\tilde{\psi}|_{P}:P\rightarrow \tilde{\psi}(P)$ is the toric variety associated to the relative star $\mathrm{Star}_\sigma(\tau')$. Let $\{\tau'_1,...,\tau'_j\}\subset \Sigma'_{\sigma}$ be the 1-dimensional cones whose corresponding divisors intersect $P$, then $\{\tau'_1,...,\tau'_j\}$ corresponds to the 1-dimensional cones in $\mathrm{Star}_\sigma(\tau')$ whose corresponding divisors intersect $P_g$. Because $P_g$ is toric, its singular locus is torus invariant. By adjunction, $\mathrm{Diff}_{P}(D')|_{P_g}$ is supported on the union of $\Supp D'\cap P_g$ and the singular locus of $P_g$, then $\Supp\mathrm{Diff}_{P}(D')|_{P_g}$ is torus invariant.
			
			By adjunction, $(P_g,\mathrm{Diff}_{P}(D')|_{P_g})$ is sub-log canonical and $K_{P_g}+\mathrm{Diff}_{P}(D')|_{P_g}\sim_{\mathbb{Q}}0$. Because $P_g$ is a general fiber, by the property of toric variety, we have $\mathrm{coeff}_{D_P}(\mathrm{Diff}_{P}(D'))=1$ for every $\tilde{\psi}|_{P}$-horizontal prime divisor $D_P$ of $\Supp\mathrm{Diff}_{P}(D')$.
		\end{proof}
	\end{lemma}
	
	\begin{lemma}\label{toric blow up is generically trivial}
		With the same notation as in Lemma \ref{toric singularities}. Suppose $P$ is a log canonical place of $(X_\Sigma,D)$, then there is an open subset $U\subset Z:=\tilde{\psi}(P)$ such that $(P,\mathrm{Diff}_{P}D')\times _{Z}U$ is crepant birationally equivalent to $(\mathbb{P}^r,D_{\Delta_r})\times U$, where $\mathbb{P}^r$ is the projective space of dimension $r:=\mathrm{dim}P-\mathrm{dim}Z$ also as the toric variety associated to the standard simplex $\Delta_r$, $D_{\Delta_r}$ is the corresponding toric boundary.
		\begin{proof}
			Recall that by Theorem 3.1, the general fiber of $f|_P:P\rightarrow Z$ is the toric variety corresponding to the relative star $\mathrm{Star}_{\sigma}(\tau')$.
			
			Let $\{\tau'_1,...,\tau'_{r}\}$ be a base of $\psi^{-1}((N_\sigma))/(N_{\tau'})$ over $\mathbb{Z}$, define $\tau'_{r+1}:=-(\tau'_1+...+\tau'_{r})$. Let $\tau_1,...,\tau_{r+1}$ be lifts of $\tau'_1,...,\tau'_{r+1}$ located in the cone $\sigma$ and $\Sigma''$ be a subdivision of $\Sigma'$ by adding the 1-dimensional cones $\tau_1,...,\tau_{r+1}$. Then there is a morphism $f:X_{\Sigma''}\xrightarrow{g} X_{\Sigma'}\xrightarrow{\tilde{\psi}}  X_{\Sigma}$. Let $\tau''$ be the preimage of $\tau'$ in $\Sigma''$, by construction, the relative star $\mathrm{Star}_{\sigma}(\tau'')$ contains the fans generated by $\{\tau'_1,...,\tau'_{r+1}\}$. By construction, the fans generated by $\{\tau'_1,...,\tau'_{r}\}$ is the standard $r$-complex $\Delta_{r}$, whose corresponding toric variety is $\mathbb{P}^r$. Then $\mathrm{Star}_{\sigma}(\tau'')$ is both a subdivision of $\Delta_n$ and $\mathrm{Star}_{\sigma}(\tau')$, and there is two natural birational morphisms $X_{\mathrm{Star}_{\sigma}(\tau'')}\rightarrow X_{\mathrm{Star}_{\sigma}(\tau')}$ and $X_{\mathrm{Star}_{\sigma}(\tau'')}\rightarrow \mathbb{P}^r$.
			
			We denote the strict transform of $P$ on $X_{\Sigma''}$ still by $P$. Write $K_{X_{\Sigma''}}+P+D''\sim_{\mathbb{Q}}f^*(K_{X_{\Sigma}}+D)$. By adjunction and Lemma \ref{toric singularities}, the general fiber of $(P,\mathrm{Diff}_{P}D'') \rightarrow Z$ is $(X_{\mathrm{Star}_{\sigma}(\tau'')},D_{\mathrm{Star}_{\sigma}(\tau'')})$ and there is an open subset $U'\subset Z$ such that $P_{U'}:=P\times_Z U'\cong X_{\mathrm{Star}_{\sigma}(\tau'')}\times U'$. The birational morphism $X_{\mathrm{Star}_{\sigma}(\tau'')}\rightarrow \mathbb{P}^r$ defines a birational morphism $P_{U'}\rightarrow \mathbb{P}^r\times U'$. Let $D_{\mathbb{P}^r_{U'}}$ be the pushforward of $D''$ on $\mathbb{P}^r\times U'$, then by lemma \ref{toric singularities}, the restriction of $D_{\mathbb{P}^r_{U'}}$ to the general fiber of $\mathbb{P}^r\times U'\rightarrow U'$ is $D_{\Delta_r}$, which is the union of $r+1$ hyperplanes in $\mathbb{P}^r$ with transversal self intersections.
			
			Next we show that there is an open subset $U\subset U'$ such that $(\mathbb{P}^r\times U',D_{\mathbb{P}^r_{U'}})\times_{U'} U\cong (\mathbb{P}^r,D_{\Delta_r})\times U$. 
			Denote the two projections by $p:\mathbb{P}^r\times U'\rightarrow \mathbb{P}^r$ and $\pi:\mathbb{P}^r\times U'\rightarrow U'$. Define $L:=p^*\mathcal{O}_{\mathbb{P}^r}(1)\otimes \pi^* \mathcal{O}_{U'}$, then $L$ is $\pi$-very ample and $\pi_*L$ is a free sheaf of rank $r+1$. Choose an open subset $U\subset U'$ such that $D_{\mathbb{P}^r_{U'}}\times _{U'}U$ form a base of section of $\pi_*L$, then these sections define a trivialization $(\mathbb{P}^r\times U',D_{\mathbb{P}^r_{U'}})\times_{U'} U\cong (\mathbb{P}^r,D_{\Delta_r})\times U$.
			
			Finally because $K_{P_U}+\mathrm{Diff}_PD'|_{P_U}\sim_{\mathbb{Q},U} 0 $ and $K_{\mathbb{P}^r\times U}+D_{\mathbb{P}^r_{U'}}|_{\mathbb{P}^r\times U}\sim_{\mathbb{Q},U}0$ and $D_{\mathbb{P}^r_{U'}}$ is the pushforward of $D''$ on $\mathbb{P}^r\times U'$, then $(P_U,\mathrm{Diff}_PD')$ is crepant birationally equivalent to $(\mathbb{P}^r,D_{\Delta_r})\times U$.
		\end{proof}
	\end{lemma}
	
	Next, we generalize Lemma \ref{log canonical center is projective bundle} to the log pair case and study the boundary part. 
	
	\begin{lemma}\label{toroidal blow up is generally trivial}
		Suppose $X$ is a normal variety, $U_X\subset X$ is a toroidal embedding and there is a $\mathbb{Q}$-divisor $D$ such that $(X,D)$ is sub-log canonical and $\Supp D\subset X\setminus U_X$. Let $f:Y\rightarrow X$ be a birational morphism extracting a log canonical place $P$ of $(X,D)$, write $K_Y+P+D_Y\sim_{\mathbb{Q}}f^*(K_X+D)$. Define $Z:=f(P)$ and denote the normalization of $P$ and $Z$ by $P^n,Z^n$. Then we have
		\begin{itemize}
			\item there is a $\mathbb{Q}$-divisor $D_{Z^n}$ on $Z^n$ such that $K_{P^n}+\mathrm{Diff}_{P^n}(D_Y)\sim_{\mathbb{Q}}(f|_P)^*(K_{Z^n}+D_{Z^n})$, and
			\item $(P^n,\mathrm{Diff}_{P^n}(D_Y))$ is crepant birationlly equivalent to $(Z^n,D_{Z^n})\times (\mathbb{P}^r,D_{\Delta^r})$, where $\mathbb{P}^r$ is the projective space of dimension $r:=\mathrm{dim}P-\mathrm{dim}Z$ also as the toric variety associated to the standard simplex $\Delta_r$, and $D_{\Delta_r}$ is the corresponding toric boundary.
		\end{itemize}
		
		\begin{proof}
			Pick a general point $x\in Z$, by the proof of Lemma \ref{log canonical center is projective bundle}, locally near $x$ we have the following Cartesian diagram
			$$
			\xymatrix{
				X'  \ar[r]^{\pi'} \ar[d]_{h} &  X_{\Sigma'} \ar[d]^{h_\sigma} \\
				X \ar[r]_\pi     &  X_{\Sigma}
			}
			$$
			such that 
			\begin{itemize}
				\item $X_\Sigma $ and $X_{\Sigma'}$ are toric varieties,
				\item $h_\sigma$ is a birational toric morphism,
				\item $\pi'$ and $\pi$ are \'etale, and
				\item $h$ is a birational morphism extracting $P$.
			\end{itemize}
			
			Because $\pi$ is \'etale, there is a $\mathbb{Q}$-divisor $D_{\Sigma}$ on $X_\Sigma$ such that $K_X+D\sim_{\mathbb{Q}}\pi^*(K_{X_\Sigma}+D_\Sigma)$ near $x$, and by the property of toroidal embedding, $\Supp D_{\Sigma}$ is torus invariant. Let $Z_\Sigma:=\pi(Z)$, by the Hurwitz's formula $Z_\Sigma$ is a log canonical center of $(X_\Sigma,D_\Sigma)$. Let $P_\Sigma$ be the divisor on $\Sigma'$ dominated by $ P$, write $K_{\Sigma'}+P_\Sigma+D'_\Sigma\sim_{\mathbb{Q}}h_\sigma^*(K_{X_\Sigma}+D_\Sigma)$. By Lemma \ref{toric blow up is generically trivial}, there is an open subset $U_\Sigma\subset Z_\Sigma$ such that $(P_\Sigma,\mathrm{Diff}_{P_\Sigma}(D'_\Sigma))\times_{Z_\Sigma} U_\Sigma $ is crepant birationally equivalent to $ (\mathbb{P}^r,D_{\Delta_r})\times U_\Sigma$.
			
			Write $K_{X'}+P+D'\sim_{\mathbb{Q}}h^*(K_X+D)$, then we have $K_{X'}+P+D'\sim_{\mathbb{Q}}\pi'^* h_\sigma^*(K_{X_\Sigma}+D_\Sigma)\sim_{\mathbb{Q}}\pi'^*(K_{X_\Sigma'}+P_\Sigma+D'_\Sigma)$. Restrict both side on $P^n$ and $P_\Sigma$, we have $K_{P^n}+\mathrm{Diff}_{P^n}(D')\sim_{\mathbb{Q}}\pi_{P} ^*(K_{P_\Sigma}+\mathrm{Diff}_{P_\Sigma}(D'_\Sigma))$. Let $U:=\pi^{-1}(U_\Sigma)$, because the diagram is Cartesian, $(P^n,\mathrm{Diff}_{\mathbb{P}^n}(D'))\times_Z U$ is isomorphic to $(P_\Sigma,\mathrm{Diff}_{P_\Sigma}(D'_\Sigma))\times_{Z_\Sigma} U$, hence crepant birationally equivalent to $(\mathbb{P}^r,D_{\Delta_r})\times U$.
			
			Because $K_{P^n}+\mathrm{Diff}_{P^n}(D')\sim_{\mathbb{Q},Z^n}0$, by the canonical bundle formula, there exists a generalized pair $(Z^n,D_{Z^n}+M)$ such that 
			$$K_{P^n}+\mathrm{Diff}_{P^n}(D')\sim_{\mathbb{Q},Z^n}(h|_{P^n})^*(K_{Z^n}+D_{Z^n}+M).$$
			Since $(P^n,\mathrm{Diff}_{P^n}(D'))\rightarrow Z^n$ is generically trivial, $M\sim_{\mathbb{Q}}0$. 
			
			Let $p_1$ denote the projection $\mathbb{P}^r\times Z^n\rightarrow \mathbb{P}^r$ and $p_2$ denote the projection $\mathbb{P}^r\times Z^n\rightarrow Z^n$. Since $K_{\mathbb{P}^r\times Z^n}=p_1^*K_{\mathbb{P}^r}+p_2^*K_{Z^n}$, we have $(\mathbb{P}^r\times Z^n,D_{\Delta_r}\times Z^n)$ is crepant birationally equivalent to $(\mathbb{P}^r,D_{\Delta_r})\times Z^n$, and because $(P^n,\mathrm{Diff}_{\mathbb{P}^n}(D'))\times_Z U$ is crepant birationally equivalent to $(\mathbb{P}^r,D_{\Delta_r})\times U$, then $(P^n,\mathrm{Diff}_{\mathbb{P}^n}(D'))$ is crepant birationally equivalent to $(\mathbb{P}^r,D_{\Delta_r})\times (Z^n,D_{Z^n})$.
		\end{proof}
	\end{lemma}

	\begin{proof}[Proof of Theorem \ref{crepant birational boundedness}]
		We use the same notation as in the proof of Theorem \ref{Main theorem, flat morphism}. Let $h:\bar{X}\rightarrow \bar{Y}$ be a birational morphism such that $\bar{P}$ is a divisor on $\bar{X}$, write $K_{\bar{X}}+\bar{\Lambda}\sim_{\mathbb{Q}}h^*(K_{\bar{Y}}+\bar{D}_{\bar{Y}}+\bar{Y}_{\bar{0}})$. We claim $(\bar{X},\bar{\Lambda})$ is the desired sub-pair.
		
		It is easy to see $(\bar{X},\bar{\Lambda})$ satisfies all the properties except the last one, so we only need to show that $(\bar{P},\mathrm{Diff}_{\bar{P}}(\bar{\Lambda}-\bar{P}))$ is crepant birationally bounded. 
		By Lemma \ref{toroidal blow up is generally trivial}, we have $(\bar{P},\mathrm{Diff}_{\bar{P}}(\bar{\Lambda}-\bar{P}))$ is crepant birationally equivalent to $(\mathbb{P}^r,D_{\Delta_r})\times (Z^n,D_{Z^n})$. Because $(\mathbb{P}^r,D_{\Delta_r})$ is fixed, we only need to show that $(Z^n,D_{Z^n})$ is log bounded.
		
		Because $K_{\bar{X}}+\bar{\Lambda}\sim_{\mathbb{Q}}h^*(K_{\bar{Y}}+\bar{D}_{\bar{Y}}+\bar{Y}_{\bar{0}})$, by adjunction, $Z^n$ is a log canonical center of $(\bar{Y},\bar{D}'_{\bar{Y}})|_{\bar{Y}_{\bar{0}}}\cong (\bar{\Cy}_s,\bar{\Cd}'_{\bar{\Cy}_s})$.
		Then apply Kawamata's subadjunction theorem on $Z^n$, we have $(K_{\bar{\Cy}_s}+\bar{\Cd}_{\bar{\Cy}_s})|_{Z^n}=K_{Z^n}+D_{Z^n}$. Because
		by Lemma \ref{lc centers form a bounded family}, $Z^n$ is in a bounded family, then the sub-adjunction shows that $(Z^n,D_{Z^n})$ is log bounded.

	\end{proof}

\end{document}